\documentclass[10pt]{amsart}

\usepackage[pagewise]{lineno}
\usepackage{amsfonts,amssymb,amsmath,amsthm}
\usepackage{graphicx}
\usepackage{xypic}
\usepackage{indentfirst}
\usepackage{enumerate}
\usepackage{setspace}
\usepackage{tikz-cd}
\usepackage{hyperref}
\usepackage{comment}

\usepackage{datetime}

\usepackage{pgfplots}
\pgfplotsset{width=10cm,compat=1.9}

\usepackage{enumitem} 
\usepackage[left=1.3in, right=1.3in, top=1.2in, bottom=1.2in]{geometry}
\usepackage[toc]{appendix}
\usepackage{mathtools}
\usepackage{mathabx,epsfig}
\usepackage{mathrsfs} %mathscr font

\usepackage{xy}
\input xy
\xyoption{all}

\newcommand{\Z}{\mathbb{Z}}
\newcommand{\uZ}{\underline{\mathbb{Z}}}
\newcommand{\F}{\mathbb{F}_2}
\newcommand{\uF}{\underline{\mathbb{F}_2}}
\newcommand{\ai}{a_{\lambda_i}}
\newcommand{\az}{a_{\lambda_0}}
\newcommand{\ao}{a_{\lambda_1}}
\newcommand{\anmt}{a_{\lambda_{n-2}}}
\newcommand{\aal}{a_{\alpha}}
\newcommand{\ui}{u_{\lambda_i}}
\newcommand{\uz}{u_{\lambda_0}}
\newcommand{\uo}{u_{\lambda_1}}
\newcommand{\unmt}{u_{\lambda_{n-2}}}
\newcommand{\ual}{u_{\alpha}}

\newcommand{\Si}{\Sigma^{-1}}

\newcommand{\bs}{\bigstar}
\newcommand{\bsp}{\bs_{\perp \lambda_0}}

\newcommand{\Hwt}{\widetilde{H}}

\theoremstyle{plain}

\newtheorem{Thm}{Theorem}
\newtheorem*{Thm*}{Theorem}
\newtheorem{Lem}[Thm]{Lemma}
\newtheorem{Prop}[Thm]{Proposition}
\newtheorem{Cor}[Thm]{Corollary}

\newtheorem{Rem}{Remark}

	\title{The $RO(C_{2^n})$-graded homotopy of $H\underline{\mathbb{F}_2}$}
	\author{Guoqi Yan}
	\address{}
	\email{gyan@nd.edu}

\date{\today}

\begin{document}

\begin{abstract}
 We give an explicit formula for the $RO(C_{2^n})$-graded homotopy of $H\uF$.
\end{abstract}

\maketitle
\tableofcontents

\section{Introduction}
$RO(G)$-graded homotopy groups of genuine equivariant spectra are known to be very difficult to compute. Even in the case of a cyclic group $G=C_{p^n}$, no inductive formula is known. Explicit formulas are only known for $n=2$, see \cite{Zeng17},\cite{NickG} and \cite{Yan22}.

In this paper we provide an explicit inductive formula for $RO(C_{2^n})$-graded homotopy of $H\uF$, where $\uF$ is the constant Mackey functor for $\F$. The main theorem of this paper is Theorem \ref{thm}. There are mainly two reasons for the $RO(C_{2^n})$-graded homotopy of $H\uF$ to be computable: Firstly, we can avoid additive extentions by working over the field $\F$. Secondly, the Mackey functor $\uF$ is constant, which enables us to deduce information from quotient groups.

The result in this paper adds new computations to the database of equivariant computations, where we can say something in general for a particular family of groups (cyclic $2$-groups in this paper). Previously, the only such computations that people know for a certain family of groups, not for a specific group, is the computation of the $RO(C_2^{\times n})$-graded homotopy of $H\uF$, by Holler and Kriz \cite{HollerKriz}.

There could be several applications of our results. In \cite{HK01}, Hu and Kriz explored Real-oriented cohomology theories. The genuine $C_2$-equivariant Steenrod algebra has several remarkable properties, including providing us a genuine $C_2$-equivariant Adams spectral sequence. The results in this paper would be an essential ingredient to the genuine $C_{2^n}$-equivariant Steenrod algebra and Adams spectral sequences. In \cite{CMayBredon},\cite{CMayRep}, Eric Hogle and Clover May discovered the freeness theorems for the group $C_2$ and coefficient $\uF$. Our result will provide information to possible generalizations of the freeness theorems. More generally, Mike Hill \cite{HillFreeness} developed a concept of $R$-free spectrum. For such a $G$-spectrum $E$, its $R$-homology will splits as the $R$-homology of induced representation spheres. When $R=H\uF$, our result will provide complete descriptions of the homology of $H\uF$-free spectra.

The main tool we will use is the following Tate square introduced in \cite{GM95}
\begin{equation}
	\xymatrix{
		H_h\ar[d]_{\simeq}\ar[r] &H\ar[d]\ar[r] &\Hwt\ar[d]\\
		H_h\ar[r] &H^h\ar[r] &H^t
	}.\label{Tate}
\end{equation}
Here $H=H\uF, H_h=EG_{+}\wedge H,H^h=F(EG_+,H),H^t=H^h\wedge\widetilde{EG}$ and $\Hwt=H\wedge\widetilde{EG}$.

Recall the real representation ring of $C_{2^n}$
\[
RO(C_{2^n})=\Z\{1,\alpha,\lambda_{n-2},\cdots,\lambda_1,\lambda_0\},
\]
where $\alpha$ is the one-dimensional sign representation, the $\lambda_k$'s are rotations of the two-dimensional real plane, by $e^{\frac{2\pi i}{2^{n-k}}}$ for each $k$, and $\Z S$ means the free abelian group generated by the set $S$. Thus the stablizer of each non-zero vector of $\lambda_k$ is $C_{2^k}$.

To illustrate the method of inducing from quotient groups, let us suppose that we already know $\pi_{\bs}^{C_{2^{n-1}}} H\F$. Let $\varepsilon:C_{2^n}\to C_{2^n}/{C_2}=C_{2^{n-1}}'$ (here we use $C_{2^{n-1}}'$ to denote the quotient group, to distinguish from the subgroup) be the canonical projection. Pulling back along $\varepsilon$ gives us a map of representation rings
\begin{equation}
    \varepsilon^*:RO(C_{2^{n-1}}')\to RO(C_{2^n})\label{varepsilon}
\end{equation}
where we have $\varepsilon^*(1)=1, \varepsilon^*(\alpha)=\alpha$ and $\varepsilon^*(\lambda_i)=\lambda_{i+1}$. Using the language of \cite[Ch.2]{LMS}, let $U$ be a complete $C_{2^n}$-universe. Let $i:U^{C_2}\to U$ be the inclusion, $\varepsilon^*:Sp^{C_{2^{n-1}}'}U^{C_2}\to Sp^{C_{2^n}}U^{C_2}$ be the functor that regards a $C_{2^{n-1}}'$-spectrum as a $C_{2^n}$-spectrum, then we have the following two adjunctions and an isomorphism:
\begin{equation}
	[i_*\varepsilon^*S^{V},H]^{{C_{2^n}}}_{U}\cong [\varepsilon^*S^{V},i^*H]^{{C_{2^n}}}_{U^{C_2}}\cong [S^{V},(i^*H)^{C_2}]^{{C_{2^{n-1}}'}}_{U^{C_2}}\cong[S^{V},H]^{{C_{2^{n-1}}'}}_{U^{C_2}}.\label{adj}
\end{equation}
Here the first adjunction is change of universe, the second is the pullback and fixed-point adjunction, and the third isomorphism comes from the fact that
\begin{align*}
(i^*H_{C_{2^n}})^{C_2}=H_{C_{2^{n-1}}'}\in Sp^{C_{2^{n-1}}'}{U^{C_2}},
\end{align*}
which is true precisely because $\uF$ is constant. Here we used the notation $H_{K}$ to denote the $K$-equivariant $H\uF$. In (\ref{varepsilon}), the image of $\varepsilon^*$ are the virtual $C_{2^n}$-representations that does not contain $\lambda_0$. If we denote
\begin{equation*}
    \bs_{\perp \lambda_0}=\Z\{1,\alpha,\lambda_{n-2},\cdots,\lambda_1\}\subset RO(C_{2^n})
\end{equation*}
to be the subgroup, then we get
\begin{equation}
    \pi_{\bsp}^{C_{2^n}}H\underset{\cong}{\xleftarrow{\varepsilon^*}} \pi_{\bs}^{C_{2^{n-1}}}H\label{induction}
\end{equation}
from (\ref{adj}). By (\ref{varepsilon}), the classes $\ai,\ui$ on the right will respectively correspond to $a_{\lambda_{i+1}},u_{\lambda_{i+1}}$ on the left, and the classes $\aal,\ual$ on the right will respectively correspond to the elements with the same name on the left. Thus, more explicitly, (\ref{induction}) means
\begin{equation*}
    \pi_{j\aal+i_{n-3}\lambda_{n-2}+\cdots+i_0\lambda_1}^{C_{2^n}}H\underset{\cong}{\xleftarrow{\varepsilon^*}} \pi_{j\aal+i_{n-3}\lambda_{n-3}+\cdots+i_0\lambda_0}^{C_{2^{n-1}}}H\quad\text{for $j,i_k\in\Z$}.
\end{equation*}
We will call (\ref{induction}) the induction formula.

\addtocontents{toc}{\protect\setcounter{tocdepth}{0}}
\subsection*{Notations} We use $\gamma$ to denote the generator of $C_{2^n}$. We use $[x]$ to denote a polynomial generator, and $\langle y\rangle  $ for an additive generator. For simplicity, all super-indices like $i,j,k$ has range $\geq 1$. For example, the notation $\F\langle \Si\frac{1}{\az^i}\rangle  [\ual^{\pm}]\langle 1,\aal\rangle$ means the classes $\Si\frac{1}{\az^i}\ual^s,\Si\frac{1}{\az^i}\ual^s\aal$ for all $i\geq 1,s\in\Z$. $\F\frac{[\ual^{\pm},\uz^{\pm}]}{[\ual,\uz]}$ will mean the quotient as a vector space. For the generators $\ai,\ui,\aal,\ual$, we refer to \cite{HHRa} for details. Our classes are the images of the classes there under the map $H\uZ\to H\uF$ except $\ual$, which is not $\uZ$-orientable but $\uF$-orientable.

\addtocontents{toc}{\protect\setcounter{tocdepth}{1}}
\section{Computations in the second row}
We start the computation from the Borel spectrum $H^h$. $EG$ has free cells in all dimensions, and its cellular chain provides a free $\Z[G]$-resolution of $\Z$. From there we derive the homotopy fixed-point spectral sequence
\begin{Thm} The $RO({C_{2^n}})$-graded homotopy fixed point spectral sequence (HFPSS) for $H\uF$ takes the form \cite[Def 1.4]{Gre18}
\begin{equation}
	E_2^{V,s}=H^s(C_{2^n};\pi_V^{C_{2^n}/e}H)\Rightarrow \pi_{V-s}^{C_{2^n}/C_{2^n}}H^h,\, |d_r|=(r-1,r)\label{HFPSS}
\end{equation}
\end{Thm}

This is a spectral sequence of algebras and it collapses at $E_2$ since $\pi_V^{G/e}H$ is concentrated in virtual representations with underlying degree $0$. The strong convergence is guaranteed by \cite[Thm 7.1]{Boa99} and the remark below it. Both the $E_2$ and the target are $\ui,\ual$-local, for all $i$, see \cite{Yan22}. The group cohomology with $\F$-coefficients are
\begin{equation*}
    \begin{aligned}
    &H^*(C_2;\F)=\F[x],|x|=1\quad\text{for $n=1$}\\
    &H^*(C_{2^n};\F)=\F[y]\otimes \Lambda\langle z\rangle  ,|y|=2,|z|=1\quad\text{for $n\geq 2$}
    \end{aligned}
\end{equation*}
where $\Lambda$ denote the exterior algebra. It is easily checked that $x$ converges to $\frac{\aal}{\ual}$ in the first case and $y,z$ converge to $\frac{\az}{\uz},\frac{\aal}{\ual}$ respectively in the second case. Recall the gold relations from \cite[Lem 3.6]{HHRb}
\begin{equation*}
    \begin{aligned}
    &a_{\lambda_i}u_{\lambda_j}=2^{i-j}a_{\lambda_j}u_{\lambda_i},\quad \text{for $n-2\geq i>  j\geq 0$}\\
    &a_{2\alpha}u_{\lambda_j}=2^{n-1-j}a_{\lambda_j}u_{2\alpha},\quad \text{for $n-2\geq j\geq 0$}
    \end{aligned}
\end{equation*}
which holds in the $C_{2^n}$-equivariant homotopy of $H\uZ$. Now since we are working with $H\uF$, these relations in $\pi_{\bs}^{C_{2^n}}H\uF$ become
\begin{equation}
    \begin{aligned}
    &a_{\lambda_i}u_{\lambda_j}=0,\quad \text{for $n-2\geq i>  j\geq 0$}\\
    &a_{2\alpha}u_{\lambda_j}=0,\quad \text{for $n-2\geq j\geq 0$},
    \end{aligned}\label{gold}
\end{equation}
which drastically simplifies our computation. We will refer these relations as the gold relations in this paper. In particular, since $H^h$ is $\ual,\ui$-local for each $i$, we have

\begin{Prop}
The $RO(C_{2^n})$-graded homotopy of $H^h$ is
\begin{equation*}
    \begin{aligned}
        &\pi_{\bs}^{C_{2^n}}H^h=\F[\aal,\ual^{\pm}]\quad\text{if $n=1$}\\
        &\pi_{\bs}^{C_{2^n}}H^h=\F[\aal,\az]/(\aal^2)[\ual^{\pm},\uz^{\pm},\cdots,u_{\lambda_{n-2}}^{\pm}]\quad\text{if $n\geq 2$}.
    \end{aligned}
\end{equation*}
\end{Prop}

When $n=1$, we have the model $S^{\infty\alpha}\simeq \widetilde{EC_{2}}$, and the maps $H\to \Hwt$ and $H^h\to H^t$ are localizations at $\aal$. When $n\geq 2$, we have the model $S^{\infty\lambda_0}\simeq \widetilde{EC_{2^n}}$. Thus the maps above are localizations at $a_{\lambda_0}$.

\begin{Prop}
The $RO(C_{2^n})$-graded homotopy of $H^t$ is
\begin{equation*}
    \begin{aligned}
    &\pi_{\bs}^{C_{2^n}}H^t=\F[\aal^{\pm},\ual^{\pm}]\quad\text{if $n=1$}\\
    &\pi_{\bs}^{C_{2^n}}H^t=\F[\aal,\az^{\pm}]/(\aal^2)[\ual^{\pm},\uz^{\pm},\cdots,u_{\lambda_{n-2}}^{\pm}]\quad\text{if $n\geq 2$}.
    \end{aligned}
\end{equation*}
\end{Prop}

Through the connecting homomorphism in the second row of (\ref{Tate}), we deduce

\begin{Cor}
The $RO(C_{2^n})$-graded homotopy of $H_h$ is
\begin{equation*}
    \begin{aligned}
    &\pi_{\bs}^{C_{2^n}}H_h=\F\langle \Si\frac{1}{\aal^i}\rangle  [\ual^{\pm}]\quad\text{if $n=1$}\\
    &\pi_{\bs}^{C_{2^n}}H_h=\F\langle \Si\frac{1}{\az^i}\rangle  [\ual^{\pm},\uz^{\pm},\cdots,u_{\lambda_{n-2}}^{\pm}]\langle 1,\aal\rangle  \quad\text{if $n\geq 2$}.
    \end{aligned}
\end{equation*}
\end{Cor}

From the above we get
\begin{Cor} We have
\begin{equation*}
    \begin{aligned}
    &\pi_{\bsp}^{C_{2^n}}H_h=\F\langle \Si(\frac{\ual}{\aal})^i\rangle   \quad\text{if $n=1$,}\\
    &\pi_{\bsp}^{C_{2^n}}H_h=\F\langle \Si(\frac{\uz}{\az})^i\rangle  [\ual^{\pm},\uo^{\pm},\cdots,u_{\lambda_{n-2}}^{\pm}]\langle 1,\aal\rangle  \quad\text{if $n\geq 2$}.
    \end{aligned}
\end{equation*}\label{perb}
\end{Cor}

\section{Sample computations for \texorpdfstring{$n=1,2$}{n12}}
The cases $n=1,2$ are too simple to fit in the general pattern in Theorem \ref{thm}. Thus we record them here. For $n\geq 1$, let $\delta_n$ be the connecting homomorphism
\[
\delta_n:\pi_{\bs}^{C_{2^n}}\Hwt\to \pi_{\bs-1}^{C_{2^n}}H_h.
\]
We use $K_n$ and $C_n$ to denote the kernel and cokernel respectively. 

Now we start with $n=1$, then
\[
\pi_{*}^{C_{2}}H_h=\F\langle \Si(\frac{\ual}{\aal})^i\rangle  
\]
and $\pi_{*}^{C_{2}}H=\F\langle 1\rangle  $. Since $H\to \Hwt$ is a ring map, the identity $1\in\pi_{0}^{C_2}H$ should map to the identity, thus not killed by elements from $\pi_{*}^{C_{2}}H_h$. We get $\pi_{*}^{C_{2}}\Hwt=\F[\frac{\ual}{\aal}]$ and thus
\begin{equation*}
    \pi_{\bs}^{C_{2}}\Hwt=\F[\ual,\aal^{\pm}]
\end{equation*}
by $\aal$-periodicity. The connecting homomorphism $\delta_1:\pi_{\bs}^{C_{2}}\Hwt\to \pi_{\bs-1}^{C_{2}}H_h$ maps
\[
\F[\ual]\langle \aal^{-i}\rangle  \mapsto \F[\ual]\langle \Si\aal^{-i}\rangle  .
\]
We get
\begin{equation*}
    \begin{aligned}
        \pi_{\bs}^{C_{2}}H=K_1\oplus C_1&=\F[\aal,\ual]\\
        &\oplus \F\langle \Si\aal^{-i}\ual^{-j}\rangle  .
    \end{aligned}
\end{equation*}

When $n=2$, we have
\begin{equation*}
    \pi_{\bsp}^{C_{4}}H_h=\F\langle \Si(\frac{\uz}{\az})^i\rangle  [\ual^{\pm}]\langle 1,\aal\rangle  
\end{equation*}
From the induction formula (\ref{induction}) we get
\begin{equation}
    \begin{aligned}
        \pi_{\bsp}^{C_{4}}H&=\F[\aal,\ual]\\
        &\oplus \F\langle \Si\aal^{-i}\ual^{-j}\rangle  .
    \end{aligned}
\end{equation}
The map $\pi_{\bsp}^{C_{4}}H_h\to \pi_{\bsp}^{C_{4}}H$ is $0$. To prove it we need a lemma

\begin{Lem}
For $n\geq 1$, and all the homotopy Mackey functor we discuss in this paper, $ker(\aal)=im(tr^{C_{2^n}}_{C_{2^{n-1}}})$ and $im(\aal)=ker(^{C_{2^n}}_{C_{2^{n-1}}})$.\label{lemrestr}
\end{Lem}
\begin{proof}
Apply $[-,S^{-V}\wedge H\uF]^{C_{2^n}}$ to the following cofiber sequence and its Spanier-Whitehead dual
\begin{equation*}
	S^{-1}\xrightarrow{\aal}S^{\alpha-1}\to {C_{2^n}}/{C_{2^{n-1}}}_+\xrightarrow{res}S^0\xrightarrow{\aal}S^{\alpha}.
\end{equation*}
\end{proof}
Now for degree reason, using divisibility and linearity, we only need to consider the possibility
\[
\Si (\frac{\uz}{\az})\ual^{-3}\aal\xrightarrow{?}\Si\aal^{-1}\ual^{-1}.
\]
Now both classes are killed by $\aal$, thus hit by $tr^{C_4}_{C_2}$. They both live in degree $\pi_{2-2\alpha}$, and $\ual$ is a unit on the $C_2$-level. Thus the problem reduces to whether the map $\pi_0^{C_2}H_h\to \pi_{0}^{C_2}H$ is $0$ or not. We know from the $C_2$ case that this map is $0$. We conclude that the map $\pi_{\bsp}^{C_{4}}H_h\to \pi_{\bsp}^{C_{4}}H$ is $0$. Then we have
\begin{equation*}
    \pi_{\bs}^{C_{4}}\Hwt=(\Sigma\pi_{\bsp}^{C_{4}}H_h)[\az^{\pm}]\oplus \pi_{\bsp}^{C_{4}}H[\az^{\pm}].
\end{equation*}
by $\az$-periodicity of $\Hwt$. For the connecting homomorphism, we first look at classes in $(\Sigma\pi_{\bsp}^{C_{4}}H_h)[\az^{\pm}]$. We get
\begin{equation*}
\F\langle \uz^i\rangle  [\az][\ual^{\pm}]\langle 1,\aal\rangle  \in K_2,
\end{equation*}
and quotient by classes of $(\Sigma\pi_{\bsp}^{C_{4}}H_h)[\az^{\pm}]$ leave us with
\begin{align}
\label{C4eq1}
&\F\langle \Si\az^{-i}\rangle  [\ual^{\pm}]\langle 1,\aal\rangle  \\
\label{C4eq2}
\oplus&\F\langle \Si\az^{-i}\uz^{-j}\rangle  [\ual^{\pm}]\langle 1,\aal\rangle  
\end{align}
in $C_2$. Elements in (\ref{C4eq2}) cannot be hit under $\delta_2$ because of the negative powers of $\uz$. The classes
\[
\F\langle \Si\az^{-i}\rangle  [\ual]\langle 1,\aal\rangle  
\]
are hit under $\delta_2$ by classes from $\pi_{\bsp}^{C_{4}}H[\az^{\pm}]$, and we are left with
\[
\F\langle \Si\az^{-i}\rangle  \frac{[\ual^{\pm}]}{[\ual]}\langle 1,\aal\rangle  \in C_2,
\]
and
\begin{equation*}
    \begin{aligned}
        &\F[\aal,\ual][\az]\\
        \oplus&\F[\ual]\langle \frac{\aal^{i+1}}{\az^j}\rangle  \\
        \oplus&\langle \Si\aal^{-i}\ual^{-j}\rangle  [\az^{\pm}]
    \end{aligned}
\end{equation*}
in $K_2$. In summary, we have
\begin{equation}
    \begin{aligned}
        \pi_{\bs}^{C_4}H=&\F[\aal,\ual,\az,\uz]/(\aal^2\uz)\\
        \oplus&\F\langle \uz^i\rangle  [\az]\frac{[\ual^{\pm}]}{[\ual]}\langle 1,\aal\rangle  \\
        \oplus&\F[\ual]\langle \frac{\aal^{i+1}}{\az^j}\rangle  \\
        \oplus&\F\langle \Si\aal^{-i}\ual^{-j}\rangle  [\az^{\pm}]\\
        \oplus&\F\langle \Si\az^{-i}\uz^{-j}\rangle  [\ual^{\pm}]\langle 1,\aal\rangle  \\
        \oplus&\F\langle \Si\az^{-i}\rangle  \frac{[\ual^{\pm}]}{[\ual]}\langle 1,\aal\rangle  .
    \end{aligned}\label{C4}
\end{equation}

\section{Statement and proof of the main theorem}
For $x\in\pi_{\bs}^{C_{2^n}}H$, let $D(x,n)$ be the set of classes in $\pi_{\bs}^{C_{2^n}}H$ which are infinitely divisible by $x$. Equivalently, it is the kernel of the algebraic completion map $\pi_{\bs}^{C_{2^n}}H\to (\pi_{\bs}^{C_{2^n}}H)^{\wedge}_x$. For example, $D(\aal,1)$ is the set of classes
\[
\F\langle \Si\aal^{-i}\ual^j\rangle  .
\]
$D(\az,2)$ include the last four summands in (\ref{C4}) as well as the classes $\F[\ual,\az]\langle \aal^i\rangle  _{i\geq 2}$ in the positive cone.
\begin{Thm}
    For $n\geq 3$, the $RO(C_{2^n})$-graded homotopy of $H\uF$ has the structure as follows. It is the direct sum of the following summands (we refer to them as part $(1)-(4)$)
    \begin{enumerate}
        \item The positive cone, denoted by $\pi_{pos}^{C_{2^n}}H$;
        \item $D(\ao,n)-\pi_{pos}^{C_{2^n}}H$;
        \item $D(\az,n)-D(\ao,n)-\pi_{pos}^{C_{2^n}}H$;
        \item The classes that are not in $D(\az,n)$.
    \end{enumerate}
    There are $n^2+2n-2$ summands in the above presentation. Here the minus sign means taking the difference of the sets involved, where we identify vector spaces with their sets of basis. Each part has explicit formulas as follows. For a closed formula of part $(2)$, see Remark \ref{Rem2}.
    \begin{enumerate}
        \item The positive cone of $\pi_{\bs}^{C_{2^n}}H$ is
    \begin{equation*}        \pi_{pos}^{C_{2^n}}H=\F[\az,\ao,\cdots,\anmt,\aal,\uz,\uo,\cdots,\unmt,\ual]/(a_{\lambda_i}u_{\lambda_j}=0,\text{for $n-1\geq i\rangle  j\geq 0$}).
    \end{equation*}
    Here for simplicity we used the notation $a_{\lambda_{n-1}}=a_{2\alpha}$ in the gold relation.\\
    \item $D(\ao,n)$ can be computed inductively as $D(\ao,n)=\varepsilon^*(D(\az,n-1))[\az^{\pm}]$. Here $\varepsilon^*$ is the map in (\ref{induction}) which identifies $\pi_{\bs}^{C_{2^{n-1}}}H$ as a subset of $\pi_{\bsp}^{C_{2^n}}H$ by renaming.\\
    \item The following $2n$ summands consists of the direct sum of the following $3$ blocks
\begin{equation*}
    \begin{aligned}
        B_1&=\F[\ual]\langle \frac{\aal^{i+1}}{\az^j}\rangle  \\
        &\oplus \F\langle \az^{-i}\rangle  (\frac{[\ao,\cdots,\anmt]}{\langle 1\rangle  }[\uo,\cdots,\unmt][\aal])/\{\text{gold relations}\},\\
        B_2&=\F\langle \Si\az^{-i}\rangle  \langle 1,\aal\rangle  \langle \uz^{-j}\rangle  [{u_{\lambda_1}}^{\pm},\cdots,\unmt^{\pm},\ual^{\pm}]\\
        &\oplus\F\langle \Si\az^{-i}\rangle  \langle 1,\aal\rangle  \langle \uo^{-j}\rangle  [{u_{\lambda_2}}^{\pm},\cdots,\unmt^{\pm},\ual^{\pm}]\\
        &\oplus\cdots\\
        &\oplus\F\langle \Si\az^{-i}\rangle  \langle 1,\aal\rangle  \langle \unmt^{-j}\rangle  [\ual^{\pm}]\\
        &\oplus\F\langle \Si\az^{-i}\rangle  \langle 1,\aal\rangle  \langle \ual^{-j}\rangle  \\
        B_3&=\F\frac{[\ual^{\pm}]}{[\ual]}\langle \unmt^i\rangle  \frac{[\ao]}{\langle 1\rangle  }[\az^{\pm}]\langle 1,\aal\rangle  \\
        &\oplus \F\frac{[u_{\lambda_{n-2}}^{\pm},\ual^{\pm}]}{[u_{\lambda_{n-2}},\ual]}\langle u_{\lambda_{n-3}}^i\rangle  \frac{[\ao]}{\langle 1\rangle  }[\az^{\pm}]\langle 1,\aal\rangle  \\
        &\oplus \cdots\\
        &\oplus \F\frac{[\unmt^{\pm},\cdots,u_{\lambda_2}^{\pm},\ual^{\pm}]}{[\unmt,\cdots,u_{\lambda_2},\ual]}\langle \uo^i\rangle  \frac{[\ao]}{\langle 1\rangle  }[\az^{\pm}]\langle 1,\aal\rangle  .
    \end{aligned}
\end{equation*}
    \item The following $(n-1)$ summands,
\begin{equation*}
    \begin{aligned}
        &\F\frac{[\ual^{\pm}]}{[\ual]}\langle \unmt^i\rangle  [\az]\langle 1,\aal\rangle  \\
        \oplus&\F\frac{[\ual^{\pm},\unmt^{\pm}]}{[\ual,\unmt]}\langle u_{\lambda_{n-3}}^i\rangle  [\az]\langle 1,\aal\rangle  \\
        \oplus&\cdots\\
        \oplus&\F\frac{[\ual^{\pm},\unmt^{\pm},\cdots,\uo^{\pm}]}{[\ual,\unmt,\cdots,\uo]}\langle \uz^i\rangle  [\az]\langle 1,\aal\rangle  .
    \end{aligned}
\end{equation*}
    \end{enumerate}\label{thm}
\end{Thm}

\begin{Rem}
Each summand in part $(4)$ has to modulo the non-negative powers of the $u$'s since they are included in the positive cone. Also notice that the $n=2$ case can also fit in our theorem if we regard the second summand of $B_1$ as $0$.
\end{Rem}

\begin{Rem}
Part $(2)$ can also be expressed as closed formulas. For each $n$, $D(\az,n)$ consists of part $(2),(3)$ and $\az^m$ times of elements in $B_1$ for suitable $m$ such that the resulting elements live in the positive cone. These are the elements that will make up part $(2)$ (and part of $\pi_{pos}^{C_{2^{n+1}}}$) for the group $C_{2^{n+1}}$. If we use ${Part}_{(i)}^{C_{2^n}},i=2,3$ to denote the part $(2)$ and $(3)$ for $C_{2^n}$. Then the closed formula will be
\begin{equation*}
    \begin{aligned}
        {Part}_{(2)}^{C_{2^n}}=&((\varepsilon^*)^{n-3}{Part}_{(2)}^{C_8})[a_{\lambda_{n-4}}^{\pm},\cdots,\az^{\pm}]\\
        \oplus&((\varepsilon^*)^{n-3}{Part}_{(3)}^{C_8})[a_{\lambda_{n-4}}^{\pm},\cdots,\az^{\pm}]\\
        \oplus&((\varepsilon^*)^{n-4}{Part}_{(3)}^{C_{16}})[a_{\lambda_{n-5}}^{\pm},\cdots,\az^{\pm}]\\
        \oplus&\cdots\\
        \oplus&((\varepsilon^*)^{2}{Part}_{(3)}^{C_{2^{n-2}}})[a_{\lambda_{1}}^{\pm},\az^{\pm}]\\
        \oplus&((\varepsilon^*){Part}_{(3)}^{C_{2^{n-1}}})[\az^{\pm}]
    \end{aligned}\label{Rem2}
\end{equation*}
for $n\geq 4$. Here $(\varepsilon^*)^i$ means the iteration of the renaming process.
\end{Rem}

We will prove the theorem by induction. First recall the definition of the positive cone
\begin{equation*}
    \pi_{pos}^{C_{2^n}}H=\underset{n\in\Z,m_i\geq 0}{\oplus}\pi^{C_{2^n}}_{n-m_0\lambda_0-m_1\lambda_1-\cdots-m_{n-2}\lambda_{n-2}-m_{n-1}\alpha}H.
\end{equation*}
It is a subring of $\pi_{\bs}^{C_{2^n}}H$ and has a particularly easy description. The entire homotopy group is an algebra over this subring.

\begin{Lem}
We have $D(\ao,n)\subset D(\az,n)$.
\end{Lem}
\begin{proof}
    As mentioned before, $\ao=0$ in $\pi^{C_{2^n}}_{\bs}H_h$. Since classes in $\pi^{C_{2^n}}_{\bs}H$ are from a quotient (also a subspace since we are working with vector spaces) of $\pi^{C_{2^n}}_{\bs}H_h$ or a subspace of $\pi^{C_{2^n}}_{\bs}\Hwt$, if we take $x\in D(\ao,n)$, it cannot come from $\pi^{C_{2^n}}_{\bs}H_h$. Thus $x$ is lifted from $\pi^{C_{2^n}}_{\bs}\Hwt$, where $x$ generates an $\az$-tower $\langle x\rangle  [\az^{\pm}]$. Again, since these classes are divisible by $\ao$, it cannot hit anything under the connecting homomorphism $\delta_n$. We deduce that if $x\in D(\ao,n)$, then $\langle x\rangle  [\az^{\pm}]\in \pi^{C_{2^n}}_{\bs}H$, thus $x\in D(\az,n)$.
\end{proof}

\begin{Lem}
$D(\ao,n)$ can be computed inductively as $D(\ao,n)=\varepsilon^*(D(\az,n-1))[\az^{\pm}]$. Here $\varepsilon^*$ is the map in (\ref{induction}) which identifies $\pi_{\bs}^{C_{2^{n-1}}}H$ as a subset of $\pi_{\bsp}^{C_{2^n}}H$ by renaming.
\end{Lem}
\begin{proof}
We have already proved in the previous lemma that any $x\in D(\ao,n)$ generates an infinite $\az$-tower $\langle x\rangle  [\az^{\pm}]\in \pi^{C_{2^n}}_{\bs}H$. Let $\widebar{D(\ao,n)}\subset D(\ao,n)$ be the subset
\[
\widebar{D(\ao,n)}=\{y\in D(\ao,n)|y\, \text{does not involve the class $\az$}\}.
\]
Then $\widebar{D(\ao,n)}[\az^{\pm}]=D(\ao,n)$, and $\widebar{D(\ao,n)}=\varepsilon^*(D(\az,n-1))$.
\end{proof}

\begin{proof}(of Theorem (\ref{thm}))\\
Direct computations in the next section shows the theorem is true when $n=3$. Assume it is true for $n-1$. We prove it for $n$ as follows.

We first show that the map $\pi_{\bsp}^{C_{2^n}}H_h\to \pi_{\bsp}^{C_{2^n}}H$ is $0$. Classes on the left hand side are
\[
\pi_{\bsp}^{C_{2^n}}H_h=\F\langle \Si(\frac{\uz}{\az})^i\rangle  [\ual^{\pm},\uo^{\pm},\cdots,u_{\lambda_{n-2}}^{\pm}]\langle 1,\aal\rangle  
\]
from Corollary (\ref{perb}). They must map to classes that are infinitely divisible by $\ual,\uo,\cdots,u_{\lambda_{n-2}}$. By the induction formula (\ref{induction}), 
\begin{equation*}
    \pi_{\bsp}^{C_{2^n}}H\underset{\cong}{\xleftarrow{\varepsilon^*}} \pi_{\bs}^{C_{2^{n-1}}}H,
\end{equation*}
they translate to infinitely $\ual,\uz,\cdots,u_{\lambda_{n-3}}$-divisible classes in $\pi_{\bs}^{C_{2^{n-1}}}H$. By induction, the only infinitely $\uz$-divisible classes live in $B_2$ of part $(3)$. By an examination of degrees, as well as using divisibility and linearity, we are left to show that the maps
\begin{equation}
    \Si(\frac{\uz}{\az})^i\uo^{-i-1}\xrightarrow{?}\Si\ao^{-i}\uo^{-1}\label{perbinduction}
\end{equation}
in degrees ${(i+1)\lambda_1-3},i\geq 1$ are $0$. We show it by computing the Mackey functors $\underline{\pi_{(i+1)\lambda_1-3}^{C_{2^n}}}H=\underline{H^{3}_{C_{2^n}}}(S^{(i+1)\lambda_1};\uF),i\geq 1$. Since $S^{(i+1)\lambda_1}$ are $C_{2^n}$-CW complexes of dimension bigger than or equal to $4$, and we care only about their third cohomology, these values does not depend on $i$ and we only need to compute $\underline{H^{3}_{C_{2^n}}}(S^{2\lambda_1};\uF)$. Now we look at the top two levels of the cellular chain of fixed-point Mackey functors computing the cohomology, we have
\begin{equation*}
    \xymatrix{
    \F\ar[r]^-N\ar@/_/[d]_{1}&\F[C_{2^n}/C_2]^{C_{2^n}}\ar[r]^-{1-\gamma=0}&\F[C_{2^n}/C_2]^{C_{2^n}}\ar[r]^-{N=0}&\F[C_{2^n}/C_2]^{C_{2^n}}\ar[r]^-{1-\gamma=0}\ar@/_/[d]_{res}&\F[C_{2^n}/C_2]^{C_{2^n}}\\
    \F\ar[r]^-N\ar@/_/[u]_{0}&\F[C_{2^n}/C_2]^{C_{2^{n-1}}}\ar[r]^-{1-\gamma}&\F[C_{2^n}/C_2]^{C_{2^{n-1}}}\ar[r]^-{N=0}&\F[C_{2^n}/C_2]^{C_{2^{n-1}}}\ar[r]^-{1-\gamma}\ar@/_/[u]_{tr}&\F[C_{2^n}/C_2]^{C_{2^{n-1}}}
    }
\end{equation*}
Here $N=1+\gamma+\gamma^2+\cdots+\gamma^{2^{n-1}-1}$. We deduce $\underline{H^3}(C_{2^n}/C_{2^n})=\F\langle N\rangle  $ and $\underline{H^3}(C_{2^n}/C_{2^{n-1}})=\F\langle N\rangle  $, and $res=1,tr=0$ in the above diagram. Notice that $i_{C_{2^i}}^*\lambda_1=\lambda_1$ for $i\geq 3$, $i^*_{C_4}\lambda_1=2\alpha$ and $i^*_{C_2}\lambda_1=i^*_{e}\lambda_1=2$. Induction shows that $\underline{\pi_{2\lambda_1-3}^{C_{2^n}}}H$ is
\begin{equation*}
    \xymatrix{
    \F\ar@/_/[d]_1\\
    \F\ar@/_/[u]_0\\
    \vdots\ar@/_/[d]_1\\
    \F\ar@/_/[u]_0\ar@/_/[d]\\
    0\ar@/_/[u]\ar@/_/[d]\\
    0\ar@/_/[u]
    }
\end{equation*}
with $res^{C_{2^i}}_{C_{2^{i-1}}}=1,tr^{C_{2^i}}_{C_{2^{i-1}}}=0,i\geq 3$. Thus whether the map (\ref{perbinduction}) is $0$ or not is reduced to the $C_4$-level, which is already shown to be $0$.

Now we know that the map $\pi_{\bsp}^{C_{2^n}}H_h\to \pi_{\bsp}^{C_{2^n}}H$ is $0$. Since $\Hwt$ is $\az$-local, we get
\begin{equation}
    \pi_{\bs}^{C_{2^n}}\Hwt=(\Sigma\pi_{\bsp}^{C_{2^n}}H_h)[\az^{\pm}]\oplus \pi_{\bsp}^{C_{2^n}}H[\az^{\pm}].
\end{equation}
Here $\Sigma$ means taking the preimage under the connecting homomorphism
\[
\delta_n:\pi_{\bs}^{C_{2^n}}\Hwt\to \pi_{\bs-1}^{C_{2^n}}H_h.
\]
We use $K_n$ and $C_n$ to denote the kernel and cokernel respectively. For elements in $\pi_{\bsp}^{C_{2^n}}H_h[\az^{\pm}]$, an easy examination shows we are left with
\begin{equation}
\F\langle \uz^i\rangle  [\az][\ual^{\pm},\uo^{\pm},\cdots,u_{\lambda_{n-2}}^{\pm}]\langle 1,\aal\rangle  \in K_n,\label{posiconeandlast}
\end{equation}
and
\begin{align}
\label{firstline}
&\F\langle \Si\az^{-i}\rangle  [\ual^{\pm},\uo^{\pm},\cdots,u_{\lambda_{n-2}}^{\pm}]\langle 1,\aal\rangle  \\
\label{secondline}
\oplus&\F\langle \Si\az^{-i}\uz^{-j}\rangle  [\ual^{\pm},\uo^{\pm},\cdots,u_{\lambda_{n-2}}^{\pm}]\langle 1,\aal\rangle  
\end{align}
in $C_n$. Note classes in (\ref{secondline}) can never be hit by $\delta_n$ because of the negative powers of $\uz$. They are the first summands of $B_2$. The classes in (\ref{posiconeandlast}) splits as
\begin{equation*}
\begin{aligned}
&\F\langle \uz^i\rangle  [\az][\ual,\uo,\cdots,u_{\lambda_{n-2}}]\langle 1,\aal\rangle  \\
\oplus&\F\langle \uz^i\rangle  [\az]\frac{[\ual^{\pm},\uo^{\pm},\cdots,u_{\lambda_{n-2}}^{\pm}]}{[\ual,\uo,\cdots,u_{\lambda_{n-2}}]}\langle 1,\aal\rangle  .
\end{aligned}
\end{equation*}
The first summands together with $\pi_{\bsp}^{C_{2^n}}H[\az]$ (see the next paragraph) makes up the new positive cone for $C_{2^n}$. The second summand are the last summands in part $(4)$. Note that
\begin{equation*}
    \begin{aligned}
        \pi_{pos}^{C_{2^n}}H&=\F[\ao,\cdots,\anmt,\aal,\uo,\cdots,\unmt,\ual]/\{\text{gold relations}\}[\az]\\
        &\oplus \F\langle \uz^i\rangle  [\az][\ual,\uo,\cdots,u_{\lambda_{n-2}}]\langle 1,\aal\rangle  .
    \end{aligned}
\end{equation*}

Now we look at elements from $\pi_{\bsp}^{C_{2^n}}H[\az^{\pm}]$. Since the target of $\delta_n$ only consists of negative powers of $\az$, we have $\pi_{\bsp}^{C_{2^n}}H[\az]\in K_n$. For $\pi_{\bsp}^{C_{2^n}}H\langle \az^{-i}\rangle  $, firstly, we have the following maps
\[
[\uo,\cdots,u_{\lambda_{n-2}},\ual]\langle \az^{-i}\rangle  \langle 1,\aal\rangle  \mapsto [\uo,\cdots,u_{\lambda_{n-2}},\ual]\langle \Si\az^{-i}\rangle  \langle 1,\aal\rangle  .
\]
By the gold relation and that every spectrum in the second row of the Tate square is $\uz$-local, we see
\begin{equation}
    a_{2\alpha}=\ai=0,\quad \text{for $n-2\geq i\geq 0$}
\end{equation}
in $\pi_{\bs}^{C_{2^n}}H_h$. Thus we get the following classes in $K_n$
\begin{align}
    &\F[\ual]\langle \aal^{j}\rangle  _{j\geq 2}\langle \az^{-i}\rangle  \\
    \oplus&\langle \az^{-i}\rangle  (\frac{[\ao,\cdots,\anmt]}{\langle 1\rangle  }[\uo,\cdots,\unmt][\aal])/\{\text{gold relations}\}.
\end{align}
Here $\frac{[\ao,\cdots,\anmt]}{\langle 1\rangle  }$ denotes the augmentation ideal of the $\F$-algebra map
\[
\F[\ao,\cdots,\anmt]\to \F,\ai\mapsto 0.
\]
They are the classes in $B_1$ in part (3).

We are still left to consider part $(4)$, $B_3$ and the rest of the summands of $B_2$. By induction assumption, part $(4)$ for $C_{2^{n-1}}$ has the following $(n-2)$ summands (we renamed them through $\epsilon^*$ so that we can regard them as a subset of $\pi_{\bsp}^{C_{2^n}}H$)
\begin{equation}
    \begin{aligned}
        &\F\frac{[\ual^{\pm}]}{[\ual]}\langle \unmt^i\rangle  [\ao]\langle 1,\aal\rangle  \\
        \oplus&\F\frac{[\ual^{\pm},\unmt^{\pm}]}{[\ual,\unmt]}\langle u_{\lambda_{n-3}}^i\rangle  [\ao]\langle 1,\aal\rangle  \\
        \oplus&\cdots\\
        \oplus&\F\frac{[\ual^{\pm},\unmt^{\pm},\cdots,{u_{\lambda_2}}^{\pm}]}{[\ual,\unmt,\cdots,u_{\lambda_2}]}\langle \uo^i\rangle  [\ao]\langle 1,\aal\rangle  .
    \end{aligned}\label{B3}
\end{equation}
Since $\ao=0$ in $\pi_{\bs}^{C_{2^n}}H_h$, classes with positive powers of $\ao$ will be in $K_n$, and they are the classes that make up $B_3$. Classes in (\ref{B3}) without $\ao$, together with their counterparts in the positive cone for $C_{2^{n-1}}$ (again renamed through $\varepsilon^*$), will hit classes in (\ref{firstline}) when paired with $\az^{-i}$. The classes in (\ref{firstline}) that are killed in this way are  
\begin{equation*}
    \begin{aligned}
        &\F[\ual^{\pm}]\langle \Si\az^{-i}\rangle  \langle \unmt^i\rangle  [\ao]\langle 1,\aal\rangle  \\
        \oplus&\F[\ual^{\pm},\unmt^{\pm}]\langle \Si\az^{-i}\rangle  \langle u_{\lambda_{n-3}}^i\rangle  [\ao]\langle 1,\aal\rangle  \\
        \oplus&\cdots\\
        \oplus&\F[\ual^{\pm},\unmt^{\pm},\cdots,{u_{\lambda_2}}^{\pm}]\langle \Si\az^{-i}\rangle  \langle \uo^i\rangle  [\ao]\langle 1,\aal\rangle  .
    \end{aligned}
\end{equation*}
Now the classes $\F[\ual]\langle \az^{-i}\rangle  \langle 1,\aal\rangle  \in \pi_{pos}^{C_{2^{n-1}}}H\langle \az^{-i}\rangle  $ will also kill the corresponding classes in (\ref{firstline}) and the quotient classes are the rest summands of $B_2$. Classes in (\ref{B3}) paired with $[\az]$ make up the classes in part $(4)$.
\end{proof}

\section{More sample computations for \texorpdfstring{$n=3,4$}{n34}}
In this section, we provide the computation for $n=3,4$. The $n=3$ case is the starting case of our induction proof of theorem (\ref{thm}), thus we computed it by hand and present it here. The $n=4$ case is used to illustrate how to use our induction methods quickly to compute for bigger $n$.

\begin{Prop}
    We have
    \begin{equation}
        \begin{aligned}
            \pi_{\bs}^{C_8}H&=\F[\aal,\ao,\az,\ual,\uo,\uz]/\{\text{gold relations}\}\\
            (L2)&\oplus\F\langle \aal^{i+1}\ao^{-j}\rangle  [\ual][\az^{\pm}]\\
            (L3)&\oplus\F\langle \Si\aal^{-i}\ual^{-j}\rangle  [\ao^{\pm}][\az^{\pm}]\\
            (L4)&\oplus\F\langle \Si\ao^{-i}\uo^{-j}\rangle  [\ual^{\pm}]\langle 1,\aal\rangle  [\az^{\pm}]\\
            (L5)&\oplus\F\langle \Si\ao^{-i}\ual^{-j}\rangle  \langle 1,\aal\rangle  [\az^{\pm}]\\
            (L6)&\oplus\F\langle \aal^{i+1}\az^{-j}\rangle  [\ual]\\
            (L7)&\oplus\F\frac{[\ao]}{\langle 1\rangle  }[\uo][\aal]/\{\text{gold relations}\}\langle \az^{-j}\rangle  \\
            (L8)&\oplus\F\langle \Si\frac{1}{\az^i\uz^j}\rangle  \langle 1,\aal\rangle  [\uo^{\pm},\ual^{\pm}]\\
            (L9)&\oplus\F\langle \Si\frac{1}{\az^i\uo^j}\rangle  \langle 1,\aal\rangle  [\ual^{\pm}]\\
            (L10)&\oplus\F\langle \Si\frac{1}{\az^i\ual^j}\rangle  \langle 1,\aal\rangle  \\
            (L11)&\oplus\F\frac{[\ual^{\pm}]}{[\ual]}\langle \uo^i\rangle  \frac{[\ao]}{\langle 1\rangle  }[\az^{\pm}]\langle 1,\aal\rangle  \\
            (L12)&\oplus\F\frac{[\ual^{\pm}]}{[\ual]}\langle \uo^i\rangle  [\az]\langle 1,\aal\rangle  \\
            (L13)&\oplus\F\frac{[\uo^{\pm},\ual^{\pm}]}{[\uo,\ual]}\langle \uz^i\rangle  [\az]\langle 1,\aal\rangle  .
        \end{aligned}
    \end{equation}
\end{Prop}
In the above presentation, $(L2)-(L5)$ are part $(2)$, $(L6)-(L11)$ are part $(3)$, and $(L12)-(L13)$ are part $(4)$.

To compute for $C_{16}$, one notices that $(L2)-(L11)$, after renaming through $\varepsilon^*$, are all infinitely divisible by $\ao$ for $C_{16}$ (which is $\az$ for $C_{8}$), thus contributes to part $(2)$ of $\pi_{\bs}^{C_{16}}H$ by adding $[\az^{\pm}]$. Together with the positive cone, part $(3)$ and $(4)$, which are provided by explicit formulas, we get $\pi_{\bs}^{C_{16}}H$.

\section{Duality}
Consider the following subspace of the positive cone
\begin{equation*}
    \F[\az,\ual,\unmt,\cdots,\uz]\langle 1,\aal\rangle  .
\end{equation*}
Its direct sum with part $(4)$ is
\begin{equation}
    \begin{aligned}
        &\F[\ual]\\
        \oplus&\F[\ual^{\pm}]\langle \unmt^i\rangle  [\az]\langle 1,\aal\rangle  \\
        \oplus&\F[\ual^{\pm},\unmt^{\pm}]\langle u_{\lambda_{n-3}}^i\rangle  [\az]\langle 1,\aal\rangle  \\
        \oplus&\cdots\\
        \oplus&\F[\ual^{\pm},\unmt^{\pm},\cdots,\uo^{\pm}]\langle \uz^i\rangle  [\az]\langle 1,\aal\rangle  .
    \end{aligned}\label{dualofB2}
\end{equation}
It is not a coincidence that these classes look like the classes in $B_2$ with every element being taken its inverse and a little shift. They are related by Anderson duality. A close examination and a simple cellular calculation shows that we have
\begin{equation}
    \underline{\pi_{\lambda_0-*}^{C_{2^n}}}H=\begin{cases}
	\uF^*,\quad *=2,\\
	0,\quad \text{otherwise}.
\end{cases}
\end{equation}
with $\pi_{\lambda_0-2}^{C_{2^n}}H=\F\langle \Si\frac{1}{\az}\frac{\aal}{\ual}\rangle  $. That is, we have
\begin{equation*}
    H\uF^*\simeq \Sigma^{2-\lambda_0}H\uF.
\end{equation*}
For convenience, we denote $\Lambda=\Si\frac{1}{\az}\frac{\aal}{\ual}$.

Recall that for any injective abelian group $A$, the functor $Hom_{Ab}(-,A)$ is exact, and $X\mapsto Hom_{Ab}(\pi_{-*}(X),A)$ defines a cohomology theory on the category of $G$-spectra, thus represented by a $G$-spectrum $I_A$. The groups $\mathbb{Q}$ and $\mathbb{Q}/\mathbb{Z}$ are injective, thus we have $I_{\mathbb{Q}}$ and $I_{\mathbb{Q}/\mathbb{Z}}$. $I_{\mathbb{Z}}$ is defined by the fiber sequence
\[
I_{\mathbb{Z}}\to I_{\mathbb{Q}}\to I_{\mathbb{Q}/\mathbb{Z}}.
\]
The Anderson dual of a $G$-spectrum $E$ is defined by $I_{\mathbb{Z}}(E)=F(E,I_{\mathbb{Z}})$. We have the following result on the Anderson dual, from \cite[Prop 4.20]{Zeng17}

\begin{Prop}
For $G$-spectra $E$ and $X$, we have a short exact sequence of $RO(C_{2^n})$-graded Mackey functors
\begin{equation*}
    0\to Ext_L(E_{\bs-1}(X),\Z)\to I_{\mathbb{Z}}(E)^{\bs}(X)\to Hom_L(E_{\bs}(X),\Z)\to 0.
\end{equation*}\label{sesanderson}
\end{Prop}
Here $Ext_L,Hom_L$ are the levelwise functors. For example, the Mackey functor $Ext_L(E_{\bs-1}(X),\Z)$ is obtained by applying $Ext^1(-,\Z):Ab\to Ab$ to each level of $E_{\bs-1}(X)$ and use the contravariance in the first variable to get the restrictions and transfers.

From the easy fact that $Hom(\uF,\mathbb{Q})=\underline{0},Hom(\uF,\mathbb{Q}/{\mathbb{Z}})=\uF^*$, where $\uF^*$ is the levelwise dual of $\uF$, we get
\begin{equation*}
    I_{\Z}(H\uF)\simeq \Si H\uF^*\simeq \Sigma^{1-\lambda_0}H\uF.
\end{equation*}

In this paper we only care about the $C_{2^n}/C_{2^n}$-level of the Mackey functors in (\ref{sesanderson}). Pluging in $E=H\uF,X=S^0$, and evaluating at the $C_{2^n}/C_{2^n}$-level, we get a short exact sequence of abelian groups
\begin{equation*}
    0\to Ext^1(\pi_{\bs-1}^{C_{2^n}}H,\Z)\to \pi_{\lambda_0-1-\bs}^{C_{2^n}}H\to Hom_{Ab}(\pi_{\bs}^{C_{2^n}}H,\Z)\to 0.
\end{equation*}
Since $\pi_{\bs}^{C_{2^n}}H$ is always $2$-torsion, the last group becomes $0$ and we get an isomorphism
\begin{equation}
    Ext^1(\pi_{\bs-1}^{C_{2^n}}H,\Z)\underset{D}{\xrightarrow{\cong}} \pi_{\lambda_0-1-\bs}^{C_{2^n}}H.\label{duality}
\end{equation}
This isomorphism gives the duality between (\ref{dualofB2}) and $B_2$. For convenience, we call it $D$. We have the following readily seen duality under $D$
\begin{equation*}
    \begin{aligned}
        D(1)=\Lambda,D(\Lambda)=1, D(\aal)=\Lambda/\aal=\Si\frac{1}{\az\ual},\\
        D(\az^{k}\ual^{i_{n-1}}\unmt^{i_{n-2}}\cdots\uo^{i_1}\uz^{i_0})=\Lambda/(\az^k\ual^{i_{n-1}}\unmt^{i_{n-2}}\cdots\uo^{i_1}\uz^{i_0}),\\
        D(\aal\az^{k}\ual^{i_{n-1}}\unmt^{i_{n-2}}\cdots\uo^{i_1}\uz^{i_0})=\Lambda/(\aal\az^k\ual^{i_{n-1}}\unmt^{i_{n-2}}\cdots\uo^{i_1}\uz^{i_0}).
    \end{aligned}
\end{equation*}

\bibliographystyle{alpha}
\bibliography{bib}

\end{document}